\documentclass[12pt]{amsart}
\usepackage{amsmath, amsfonts, amssymb}
\usepackage{enumitem,dsfont}
\usepackage{color}
\usepackage{graphicx}

\newcommand{\R}{{\mathbf R}}
\newcommand{\ve}{\varepsilon}

\newcommand{\jb}[1]{\left\langle #1 \right\rangle}

\newtheorem{thm}{Theorem}
\newtheorem{lm}{Lemma}
\newtheorem{cor}[lm]{Corollary}
\newtheorem{prop}[lm]{Proposition}
\newtheorem{rem}[lm]{Remark}
\numberwithin{equation}{section}
\numberwithin{thm}{section}
\numberwithin{lm}{section}
\title [Critical exponent for nonlinear damped wave equations]
{Critical exponent for nonlinear damped wave equations with non-negative potential in 3D}

\author{Vladimir Georgiev}
\address{V. Georgiev
\newline Dipartimento di Matematica Universit\`a di Pisa
Largo B. Pontecorvo 5, 56100 Pisa, Italy\\
 and \\
 Faculty of Science and Engineering \\ Waseda University \\
 3-4-1, Okubo, Shinjuku-ku, Tokyo 169-8555 \\
Japan and IMI--BAS, Acad.
Georgi Bonchev Str., Block 8, 1113 Sofia, Bulgaria}%
\email{georgiev@dm.unipi.it}%
\thanks{ The first author was supported in part by  Project 2017 "Problemi stazionari e di evoluzione nelle equazioni di campo nonlineari" of INDAM,
GNAMPA - Gruppo Nazionale per l'Analisi Matematica,
la Probabilit\`a e le loro Applicazioni,
by Institute of Mathematics and Informatics,
Bulgarian Academy of Sciences and Top Global University Project, Waseda University,  by the University of Pisa, Project PRA 2018 49 and project "Dinamica di equazioni nonlineari dispersive", "Fondazione di Sardegna" , 2016.}
\author{Hideo Kubo}
\address{H. Kubo \newline
Department of Mathematics,
Faculty of Science, Hokkaido University,
Sapporo 060-0810, Japan}
\email{kubo@math.sci.hokudai.ac.jp}
\thanks{The second author was
partially supported by Grant-in-Aid for Science Research (No.16H06339 and No. 26220702),
JSPS}

\author{Kyouhei Wakasa}
\address{K.Wakasa \newline
Department of Mathematics, Faculty of Science and Technology, 
Tokyo University of Science, 2641 Yamazaki, Noda, Chiba 278-8510, Japan.}
\email{wakasa\_kyouhei@ma.noda.tus.ac.jp}
\thanks{The third author was partially supported by Grant-in-Aid for Scientific Research 
(No.18H01132), JSPS}

\begin{document}

\begin{abstract}
We are studying possible interaction of damping coefficients in the subprincipal part of the linear 3D wave equation and their impact on the critical exponent of the corresponding nonlinear Cauchy problem with small initial data. The main new phenomena is that certain relation between these coefficients may cause very strong jump of the critical Strauss exponent in 3D to the critical 5D Strauss exponent for the wave equation without damping coefficients.
\end{abstract}
\maketitle

\section{Introduction}
We consider the Cauchy problem for the nonlinear damped wave equation with a non-negative potential:
\begin{align} \label{1.1}
& (\partial_t^2 + 2w(|x|) \partial_t -\Delta +V(|x|))U= \lambda |U|^p
\quad \mbox{in} \ (0,T) \times {\mathbb R}^3,
\\ \label{1.2}
& U(0,x)=\varepsilon f_0(|x|), \quad (\partial_t U)(0,x)=\varepsilon f_1(|x|)
\quad \mbox{for} \ x \in {\mathbb R}^3,
\end{align}
where $f_0$, $f_1$, $w$, and $V$ are assumed to be radially symmetric functions
in ${\mathbb R}^3$.

The case without any damping term, i.e. the case when $w=V=0,$ has been intensively studied  for  few decades (see \cite{Str89}, \cite{Joh79}, \cite{G81}, \cite{GLS97}, \cite{DGK01}, \cite{YZ06}, \cite{KubOht05}, or references in \cite{G2005}) and in this case there is a critical nonlinear exponent
 known as Strauss critical exponent that separates global small data solutions and blow - up of the small data solution for finite time.
  This critical exponent $p_0(n)$ is the positive root of
$$
p \left(\frac{n-1}2 p- \frac{n+1}2 \right)=1.
$$

In case of presence the damping terms with  $w(r)=c_1/r,$ $V(r) = c_2/r^2$, where  $c_1,c_2>0$, one can  pose the question if
the interaction between damping terms and nonlineare source term  can produce appropriate shift of the Strauss exponent.

For this, we suppose that $w(r)$ is a positive decreasing function in
$C([0,\infty)) \cap C^1(0,\infty)$ satisfying
\begin{align} \label{1.3}
w(r)=\frac1{r} \quad \mbox{for} \ r \ge r_0
\end{align}
with some positive number $r_0$.  

First, we recall known results concerning the linear damped wave equations, i.e.,
$V\equiv 0$ and $\lambda=0$.
It was shown in Mochizuki \cite{Mochi76} that if
$0<w(|x|) \le C(1+|x|)^{-\alpha}$ for $x \in {\mathbb R}^3$ and
$\alpha>1$, then the scattering to the free wave equation occurs in the energy space,
without assuming the radial symmetry.
On the other hand, if
$w(x) \ge C(1+|x|)^{-\alpha}$ for $x \in {\mathbb R}^3$ and $\alpha \le 1$,
then we see from the work of Matsumura \cite{Matsu77} that the energy for the wave equation decays to zero as time goes to infinity.
In this paper we focus on the borderline case $\alpha=1$.

For the semilinear wave equation with potential
$$ (\partial_t^2  -\Delta +V(x))u= \lambda |u|^p
\quad \mbox{in} \ (0,T) \times {\mathbb R}^3,$$
one can find blow up result in \cite{ST97} or global existence part in  \cite{GHK01}.

In the case where the coefficient of the damping term is a function of  time variable,
D'Abbicco, Lucente and Reissig \cite{DAbLucRei15} derived the critical exponent for the Cauchy problem to
\begin{align} \label{eq.t1}
\left( \partial_t^2 + \frac2{1+t} \partial_t -\Delta \right)U= \lambda |U|^p
\quad \mbox{in} \ (0,T) \times {\mathbb R}^3,
\end{align}
by assuming the radial symmetry. Indeed, they proved that the problem admits a global solution for sufficiently small initial data if $p>p_0(5)$,
and that the solution blows up in  finite time if $1<p<p_0(5)$.

This result can be interpreted as a result of the action of the damped term in \eqref{eq.t1} that shifts the critical exponent for small data solutions from $p_0(3)$ to $p_0(5).$

The assumption about the radial symmetry posed in \cite{DAbLucRei15} was removed by Ikeda, Sobajima \cite{IkeSob17time} for the blow-up part
(actually, they treated more general damping term ${\mu} { (1+t)^{-1} } \partial_t u$ with $\mu>0$), and
by Kato, Sakuraba \cite{KatSak18} and Lai \cite{Lai18}
for the existence part,  independently.

Now we turn back to the case where the coefficient of the damping term is a function of spatial variabletime variavles.
Ikeda, Sobajima \cite{IkeSob17space} considered the Cauchy problem for
\begin{align} \label{eq.x1}
\left( \partial_t^2 + \frac{V_0}{|x|} \partial_t -\Delta \right)U= \lambda |U|^p
\quad \mbox{in} \ (0,T) \times {\mathbb R}^n,
\end{align}
and proved a blow-up result together with the upper bound of the lifespan, provided that
$0<V_0<(n-1)^2/(n+1)$, $n/(n-1) <p \le p_0(n+V_0)$, and that $p<(n-2)/(n-4)$ if $n \ge 5$.

We shall study the combined effect between the damping and potential terms in this paper, provied the following relation
\begin{align} \label{1.5}
V(r)=-w^\prime(r) + w(r)^2 \quad \mbox{for} \ r>0.
\end{align}
Since we assumed that $w$ is a decreasing function, we see that $V$ is a non-negative function.
Roughly speaking, our result is similar to \cite{DAbLucRei15} in the sense that the critical exponent is shifted
from $p_0(3)$ to $p_0(5)$.

Another important phenomena, closely related to the space shift of the critical exponent, is the behavior of the supercritical solution near the light cone. Indeed, for $ p>p_0(5)$ we shall see that the far field behavior of $U(t,r)$  is given by
$$ U(t,r) \lesssim \frac{1}{t^2} , \ \ r \in (t/2,t) , \ \  t \to \infty,$$  
so the decay rate for the solution $U$ to the 3D problem \eqref{1.1} is the same as the 5D linear wave equation.

This paper is organized as follows.
In the section 2, we formulate the problem and state our results in Theorems 2.1 and 2.2.
The section 3 is devoted to the proof of a blow-up result given in Theorem 2.1.
In the section 4, we derive a priori upper bounds and complete the proof of Theorem 2.2.


\section{Formulation of the problem and Results}


Since the Cauchy problem \eqref{1.1}-\eqref{1.2} is rotationally invariant,
we can make the substitution
$$
U(t,r\omega)=\frac{u(t,r)}{r} \quad \mbox{with} \ r=|x|,~ \omega=x/|x|,
$$
and obtain
\begin{align} \label{1.6}
& (\partial_t^2 + 2{w}(r) \partial_t -\partial_r^2 +{V(}r))u={|u|^p}/{r^{p-1}}
\quad \mbox{in} \ (0,T) \times (0,\infty),
\\ \label{1.7}
& u(0,r)=\varepsilon r \tilde{f}_0(r), \quad (\partial_t u)(0,r)=\varepsilon r \tilde{f}_1(r)
\quad \mbox{for} \ r>0
\end{align}
together with the boundary condition $u(t,0)=0$ for all $t \in (0,T)$.
By the relation \eqref{1.5}, we have the following factorization of the operator in \eqref{1.6}:
\begin{align} \label{1.10}
&\partial_t^2 + 2{w}(r) \partial_t -\partial_r^2 +{V}(r) = \\ \nonumber
= &(\partial_t -\partial_r+w(r)) (\partial_t +\partial_r+w(r)) \quad \mbox{for} \ r > 0.
\end{align}
This relation \eqref{1.10} suggests us to consider the following equations:
\begin{equation}\label{eq.fs1}
   P_+ v_+ = g, \quad   P_- v_- = g
\quad \mbox{in} \ (0,\infty) \times (0,\infty)
\end{equation}
with $$ P_\pm = \partial_t \pm \partial_r + w(r).
$$
Setting $W(r)=\displaystyle \int_0^r w(\tau) d\tau$ for $r \ge 0$,
\eqref{eq.fs1} can be rewitten as
\begin{align*}
& \partial_s ( e^{ W(r - s)} v_+ (t-s, r - s) ) = - e^{ W(r - s )} g(t-s, r - s ), \ 0\le s \le  \min \{t, r\},
\\
& \partial_s ( e^{ -W(r + s)} v_-(t-s, r + s) ) = - e^{ -W(r + s )} g(t-s, r + s ), \ 0\le s \le t
\end{align*}
for a fixed $(t,r)$.
Then a simple integration over $(0,T)$ gives

\begin{lm}
Let $t>0$, $r>0$.
If $v_\pm$ solves \eqref{eq.fs1}, then we have
\begin{align}\label{eq.fs3+}
   & v_+(t,r) = \\ \nonumber & = e^{-W(r)+W(r - T)}v_+(t-T,r - T)
+ \int_0^T e^{- W(r) + W(r- s)} g(t-s,r - s) ds
\end{align}
for $ 0 < T\le \min \{t, r\}$, and
\begin{align}\label{eq.fs3-}
   & v_- (t,r) = \\ \nonumber & = e^{W(r)-W(r+ T)} v_-(t-T,r + T)
+ \int_0^T e^{W(r) - W(r+ s)} g(t-s,r + s) ds
\end{align}
for $ 0 < T\le t$.
\end{lm}

Therefore, the mixed initial-boundary valued problem
\begin{align} \label{eq.fs4}
& P_- P_+ u= \,
F
\quad \mbox{in} \ (0,\infty) \times (0,\infty),
\\ \nonumber & u(t,0) =0, \quad \ \ \mbox{for} \ \ t \in (0,\infty),
\\ \nonumber
& u(0,r)= \varphi(r), \ \  (\partial_t u)(0,r)=\psi(r)
\quad \mbox{for} \ r \in (0,\infty),
\end{align}
has a solution represented via integration over
$$
\Delta_-(t,r)
=\{(\sigma, y) \in (0,\infty) \times (0,\infty); |t-r| < \sigma+y < t+r, \
\sigma-y <t-r \}
$$
with an appropriate kernel
\begin{align} \label{eq.d1}
E_-(t,r,y)=
e^{-W(r)} e^{2W( 2^{-1} (y- t+r) )} e^{-W(y)} \quad
\mbox{for}\ t, r \ge 0, ~ y \ge t-r.  
\end{align}
More precisely, we have the following assertion.

\begin{prop}
If $u$ solves \eqref{eq.fs4},
then for $0 < r < t$ we have
\begin{align} \label{eq.fs5m}
 u(t,r) 
=
& \frac12 \iint_{\Delta_-(t,r)} E_-(t-\sigma,r,y) F(\sigma, y) dy d\sigma
\\ \nonumber
& +\frac12 \int_{t-r}^{t+r} E_-(t,r,y) \left( \psi(y) + \varphi^\prime(y)+ w(y)
\varphi(y) \right) dy.
\end{align}
Furthermore, for $0 < t < r$ we have
\begin{align} \label{eq.fs5m1}
u(t,r)  
=
& \frac12 \iint_{\Delta_-(t,r)} E_-(t-\sigma,r,y) F(\sigma, y) dy d\sigma +E_-(t,r,r-t) \varphi(r-t)
\\ \nonumber
& +\frac12 \int_{r-t}^{t+r} E_-(t,r,y) \left( \psi(y) + \varphi^\prime(y)+ w(y)
\varphi(y) \right) dy.
\end{align}
\end{prop}

\begin{proof}
Setting $u_+= P_+ u$, we find
\begin{align}\label{eq.fs5}
    P_- u_+ = F, \\ \nonumber
    P_+ u=u_+.
\end{align}
Using  \eqref{eq.fs3-} with $T=t$, we get
\begin{align} \label{eq.fs5m2}
 u_+(t,r) = & e^{ W(r)- W(r + t)}u_+(0,r+t )
 \\ \notag
  & + \int_0^{t} e^{ W(r) - W(r+ s)} F(t-s,r+ s) ds.
\end{align}
Using \eqref{eq.fs3+} with $T =r$ or $T=t$, we find
$$
u(t,r)=\underbrace{e^{- W(r) }u(t-r,0)}_{=0} + \int_0^r e^{- W(r) + W(r- s)} u_+(t-s, r- s) ds
$$
for  $0 <r < t$, and
\begin{align*}
u(t,r)=& e^{- W(r) +W(r-t)} u(0,r-t) +
\\
& +\int_0^{t} e^{- W(r) + W(r- s)} u_+(t-s, r- s) ds
\end{align*}
for $0<t <r$.
Combining  these relations, we get \eqref{eq.fs5m} and \eqref{eq.fs5m1}.
Indeed, when $t>r$, we have
\begin{align*}
u(t,r) = & \int_0^r e^{- W(r) + 2W(r- s) - W(r + t-2s)}
u_+(0,r+t-2s ) ds
\\ 
& + \int_0^r \int_0^{t-s} e^{- W(r) + 2W(r- s) - W(r-s+ s^\prime)}
F(t-s-s^\prime, r- s+s^\prime) ds^\prime ds
\\
= & \frac12 \int_{t-r}^{t+r} E_-(t,r,y) u_+(0,y) dy
\\ 
& + \frac12 \iint_{\Delta_-(t,r)} E_-(t-\sigma,r,y) F(\sigma, y) dy d\sigma.
\end{align*}
Since $ u_+(0,y)= \psi(y) + \varphi^\prime(y)+ w(y) \varphi(y)$, we get
\eqref{eq.fs5m}.
On the other hand, when $0<t<r$, we have
\begin{align*}
u(t,r) = & e^{- W(r) +W(r-t)} u(0,r-t)
\\
& + \int_0^t e^{- W(r) + 2W(r- s) - W(r + t-2s)}
u_+(0,r+t-2s ) ds
\\ 
& + \int_0^t \int_0^{t-s} e^{- W(r) + 2W(r- s) - W(r-s+ s^\prime)}
F(t-s-s^\prime, r- s+s^\prime) ds^\prime ds
\\
= & E_-(t,r,r-t) u(0,r-t)+ \frac12 \int_{r-t}^{t+r} E_-(t,r,y) u_+(0,y) dy
\\ 
& + \frac12 \iint_{\Delta_-(t,r)} E_-(t-\sigma,r,y) F(\sigma, y) dy d\sigma,
\end{align*}
which implies \eqref{eq.fs5m1}.
This completes the proof.
\end{proof}

\begin{rem}
The assumption \eqref{1.3}  implies
\begin{align} \label{eq.le1}
e^{W(r)} \sim \jb{r},  \ \ r >0.
\end{align}
Then the definition \eqref{eq.d1} of $E_-$ implies
\begin{align} \label{eq.le2}
E_-(t,r,y) \sim  \frac{\langle  r-t +y \rangle^2}{\langle r \rangle \langle y \rangle}.
\end{align}
\end{rem}

\begin{rem}
Suppose that \eqref{1.3} and \eqref{1.5} hold.
If $u$ solves
\begin{align} \label{eq.lb1}
&   (\partial_t^2 + 2{w}(r) \partial_t -\partial_r^2 +{V(}r))u= \,
|u|^p/r^{p-1}
\quad \mbox{in} \ (0,T) \times  (0,\infty),
\\ \nonumber & u(t,0) =0 \quad \ \ \mbox{for} \ \ t \in (0,T),
\\ \nonumber
& u(0,r)= 0, \ \   (\partial_t u)(0,r)= \psi(r)
\quad \mbox{for} \ r \in (0,\infty),
\end{align}
then from \eqref{eq.fs5m} and \eqref{eq.fs5m1} we find the following lower bound:
\begin{align} \label{eq.lb2}
& u(t,r) \gtrsim \widetilde{I_-}(|u|^p/y^{p-1})(t,r) + \widetilde{J_-}(\psi)(t,r)
\end{align}
for $t>0$, $r>0$, where we put
\begin{align}\label{eq.lb2a}
& \widetilde{I_-}(F)(t,r) =  \iint_{\Delta_-(t,r)} \widetilde{ E_-}(t-\sigma,r,y)F(\sigma, y) dy d\sigma,
\\
\label{eq.lb2b}
& \widetilde{J_-}(\psi)(t,r) = \int_{ |t-r| }^{t+r} \widetilde{ E_-}(t,r,y) \psi(y) dy
\end{align}
with
$$ \widetilde{E_-}(t,r,y) = \frac{\langle r-t + y \rangle^2}{\langle r \rangle \langle y \rangle}.$$
\end{rem}


Now we are in a position to state our blow-up result.

\begin{thm} \label{bw2}
Suppose that \eqref{1.3} and \eqref{1.5} hold.
Let 
$\psi \in C_0(\R^3)$ be a nonzero non-negative function. If
$1< p < ({3+\sqrt{17}})/{4}=p_0(5)$,
then the classical solution to the problem \eqref{eq.lb1} blows up in finite time.

Moreover, there exists a positive constant $C^*$ independent of
$\varepsilon$ such that
\begin{eqnarray}\label{life:upper0}
T^*(\varepsilon) \le \left\{ \begin{array}{ll}
\exp(C^*\varepsilon^{-p(p-1)})
& \mbox{if} \hspace{3mm} p=p_0(5), \\
C^*\varepsilon^{-p(p-1)/(1+3p-2p^2)}
& \mbox{if} \hspace{3mm} 1<p <p_0(5).
\end{array} \right.
\end{eqnarray}
Here $T^*(\varepsilon)$ denotes the lifespan of the problem  \eqref{eq.lb1}.
\end{thm}

To show the counter part of Theorem \ref{bw2},
we introduce an integral equation associated with
the problem \eqref{1.6}-\eqref{1.7}:
\begin{align} \label{eq.io0}
u(t,r) 
=
\varepsilon u_0(t,r)+
\frac12 \iint_{\Delta_-(t,r)} E_-(t-\sigma,r,y) \frac{|u(\sigma, y)|^p}{y^{p-1} } dy d\sigma
\end{align}
for $t>0$, $r>0$, where we have set
\begin{align} \label{eq.io00}
u_0(t,r)=
& \frac12 \int_{|t-r|}^{t+r} E_-(t,r,y) \left( \psi(y) + \varphi^\prime(y)+ w(y)
\varphi(y) \right) dy
\\ \nonumber
& + \chi(r-t) E_-(t,r,r-t) \varphi(r-t)
\end{align}
with $\varphi(r)=r \tilde{f}_0(r)$, $\psi(r)=r \tilde{f}_1(r)$, where $\chi(s)=1$ for $s \ge 0$,
and $\chi(s)=0$ for $s<0$.

\begin{thm} \label{ge1}
Suppose that \eqref{1.3} and \eqref{1.5} hold.
Assume $p> ({3+\sqrt{17}) }/{4} = p_0(5)$.
Let $\tilde{f}_0 \in C^1([0,\infty))$, $\tilde{f}_1 \in C^0([0,\infty))$
satisfy
\begin{align} \label{eq.zj1d}
|\tilde{f}_0(r)| \le \langle r \rangle^{-\kappa-2}, \quad
|\tilde{f}_0^\prime (r)| +|\tilde{f}_1(r)| \le \langle r \rangle^{-\kappa-3} \ \ \mbox{for} \ r \ge 0
\end{align}
with some positive constant $\kappa \ge 2p-3$.
Then there exists $\varepsilon_0>0$ so that the corresponding integral equation
\eqref{eq.io0} to
the problem \eqref{1.6}-\eqref{1.7} has a global solution satisfying
$$
|u(t,r)| \lesssim \varepsilon r \, \langle r \rangle^{-2} \langle t-r \rangle^{-(2p-3)},
\ t>0, ~ r>0
$$
for any $\varepsilon \in (0,\varepsilon_0].$
\end{thm}


\section{Blow-up}


In this section we prove the blow-up result
in an analogous manner to \cite{Joh79} (see also \cite{Tak94} and \cite{KubOht05}).
Our first step in this subsection is to obtain basic lower bounds of the solution to the problem \eqref{eq.lb1}.

\begin{lm} \label{fit1R}
Let $R \ge 1$ and $p>1$.
We assume
\begin{align} \label{bua}
 \psi(r) >0 , \ \forall r \in (0,R), \ \ \psi(r)=0, \ \ \forall r \geq R.
\end{align}
Then we have
\begin{equation}\label{eq.le6aR}
\widetilde{J_-}(\psi)(t,r)\gtrsim \frac{c_0}{ \langle r \rangle}, \ \
 c_0 := \inf_{R/2 \le r \le 2R/3} \psi(r)
\end{equation}
for
\begin{equation}\label{eq.le6a1R}
t < r < t+({R}/{2}), ~ t+r >  R.
\end{equation}
Moreover, if $u$ is the solution to \eqref{eq.lb1}, then we have
\begin{equation}\label{eq.le7R}
\widetilde{I_-}(|u|^p/y^{p-1})(t,r)
\gtrsim \frac{c_0^p }{\langle r \rangle (t-r)^{2p-3}}
\end{equation}
for $ 0<t < 2r $ and $t-r > R$.
\end{lm}

\begin{proof}
The estimate \eqref{eq.le6aR} follows from \eqref{eq.lb2b} and our choice of $\psi.$
Indeed, if $(t,r)$ are close to the light cone as in \eqref{eq.le6a1R}, then we have
\begin{align} \notag
\widetilde{J_-}(\psi)(t,r) &
\gtrsim \int_{r-t}^{t+r} \frac{\langle r-t + y \rangle^2}{\langle r \rangle \langle y \rangle} \psi(y) dy
\\ \notag
& \gtrsim c_0 \int_{R/2}^{2R/3} \frac{1}{\langle r \rangle \langle y \rangle} dy
\gtrsim c_0 \langle r \rangle^{-1},
\end{align}
because of \eqref{eq.le6a1R}.

The estimate \eqref{eq.le7R} follows from \eqref{eq.lb2a} and \eqref{eq.le6aR},
provided $ 0<t < 2r $ and $t-r > R$.
Indeed, with $F(\sigma,y)= |u(\sigma,y)|^p/y^{p-1} $ we have
\begin{align}
\widetilde{I_-}(F)(t,r)
& =  \iint_{\Delta_-(t,r)} \widetilde{ E_-}(t-\sigma,r,y)
\frac{ |u(\sigma,y)|^p }{ y^{p-1} } dy d\sigma.
\end{align}
Since 
the domain
$$
\Sigma=\{(\sigma,y) \in (0,\infty) \times (0,\infty);
 0 \leq y-\sigma \leq R/2, \, t-r < \sigma+y <t+r \}
$$
is a subset of the integration domain
$ \Delta_-(t,r)$ for $t-r>R$,
we see from \eqref{eq.lb2} and \eqref{eq.le6aR} that
$u(\sigma, y) \gtrsim c_0 \langle y \rangle^{-1}$
for $(\sigma, y) \in \Sigma$.
Therefore, we get
\begin{align} \notag
\widetilde{I_-}(F)(t,r)
& \gtrsim c_0^p
\iint_{\Sigma}
\frac{\langle -t+\sigma +r+y \rangle^2 }{\langle r \rangle \langle y \rangle ^{2p}} dy d\sigma.
\end{align}
Now, introducing the coordinates $\alpha=\sigma+y$, $\beta=\sigma-y$, we obtain
\begin{align}
\widetilde{I_-}(F)(t,r)
& \gtrsim c_0^p \int_{t-r}^{t+r}
\frac{\langle \alpha-t+r \rangle^2 }{\langle r \rangle \langle \alpha \rangle^{2p}} d\alpha,
\end{align}
because $\beta \sim 1$.
Since $t<2r$, we have $t+r > 3(t-r)$,
so that
\begin{align} \notag
\langle r \rangle \widetilde{I_-}(F)(t,r)
& \gtrsim
c_0^p \int_{t-r}^{3(t-r)} \frac{(\alpha -t+r)^2 }{ \langle \alpha \rangle^{2p}}
d \alpha
\\ \notag
& \gtrsim
c_0^p
\langle t-r \rangle^{-2p}
\int_{t-r}^{3(t-r)} (\alpha-t+r)^2 d\alpha
\\ \notag
& \gtrsim
c_0^p
( t-r )^{-2p+3} ,
\end{align}
for $t-r>R$.
This completes the proof.
\end{proof}

For $\eta>0$, we introduce the following quantity:
\begin{eqnarray}
&& \langle u \rangle(\eta)
=\inf\{ \jb{y} (\sigma-y)^{2p-3}|u(\sigma,y)|:
(\sigma,y)\in {\Sigma}(\eta)\}, \label{inf0}
\\
&&
\Sigma(\eta)=\{(\sigma,y); \ 0 \le \sigma \le 2y, \ \sigma-y \ge \eta\}.
\label{def:Sigma}
\end{eqnarray}
Since we may assume $0<R \le 1$, \eqref{eq.lb2} and \eqref{eq.le7R} yield
\begin{eqnarray}\label{leb3}
\langle u \rangle(y) \ge C_1\varepsilon^{p} \quad \text{for}  \ y \ge 1.
\end{eqnarray}
We shall show that there exists a constant $C_2>0$ such that
\begin{equation}\label{leb4}
\langle u \rangle(\xi) \ge
C_2 \int_{1}^{\xi} \left(1-\frac{\beta}{\xi}\right)
\frac{[\langle u \rangle(\beta)]^p}
{\eta^{pp^*}}\, d\beta, \  \xi \ge 1
\end{equation}
for some $p^*>0$.
Let $\xi\ge 1$ and $(t,r) \in \Sigma(\xi)$.
For $\eta>0$ we set
$$
\tilde{\Sigma}(\eta,t-r)=\{(\sigma,y); \ y \ge t-r,\ \sigma+y \le 3(t-r), \ \sigma-y \ge \eta\}.
$$
It is easy to see that $\tilde{\Sigma}(\eta,t-r) \subset \Delta_-(t,r)$ for any $\eta>0$ and
$(t,r) \in \Sigma(\xi)$ and that
$(\sigma,y) \in \tilde{\Sigma}(1,t-r)$ implies $(\sigma,y) \in {\Sigma}(\sigma-y)$.
Therefore, we have
\begin{align*}
\widetilde{I_-}(F)(t,r)
& \gtrsim \frac{1}{\jb{r} }\iint_{\tilde{\Sigma}(1,t-r)}
\frac{(-t+\sigma+r+y)^2}{\jb{y} }
\frac{ |u(\sigma,y)|^p }{y^{p-1} }
\,dy\,d\sigma
\\
& \gtrsim \frac{(t-r)^2}{\jb{r} }\iint_{\tilde{\Sigma}(1,t-r)}
\frac{ \left[\langle u \rangle ({\sigma-y) }\right]^p }
{ \jb{y}^{2p}(\sigma-y)^{p(2p-3)}}
\,dy\,d\sigma,
\end{align*}
because $-t+\sigma+r+y \ge -t+r +(\sigma-y)+2y \ge 1+(t-r)$
for $(\sigma,y) \in \tilde{\Sigma}(1,t-r)$.
Changing the variables by $\beta=\sigma-y$, $z=y$, we have
\begin{eqnarray*}
u(t,r)
&\gtrsim & \frac{(t-r)^2 }{ \jb{r}} \int_{1}^{t-r}
\left( \int_{t-r}^{(3(t-r)-\beta)/2}
\frac{ \left[\langle u \rangle ({\beta})\right]^p }
{ \jb{z}^{2p} \beta^{p(2p-3)} } dz
\right)\,d\beta
\\
&\gtrsim & \frac{1 }{ \jb{r} (t-r)^{2p-2} } \int_{1}^{t-r}
\frac{t-r-\beta}{ 2}
\frac{ [\langle u \rangle(\beta)]^p}{\beta^{p(2p-3)}}
\,d\beta
\\
&\gtrsim & \frac{1 }{ \jb{r} (t-r)^{2p-3} } \int_{1}^{t-r}
\left( 1-\frac{\beta}{t-r} \right)
\frac{ [\langle u \rangle(\beta)]^p}{\beta^{p(2p-3)}}
\,d\beta.
\end{eqnarray*}
Since the function
$$
y\mapsto \int_{1}^{y}
\left(1-\frac{\beta}{y}\right)
\frac{[\langle u \rangle(\beta)]^p}
{\beta^{p(2p-3)}}\,d\beta
$$
is non-decreasing,
for any $(t,r)\in {\Sigma}(\xi)$, we have
$$ \jb{r} (t-r)^{p^*} u(t,r)
\ge
C_2 \int_{1}^{\xi}\left(1-\frac{\beta}{\xi}\right)
\frac{[\langle u \rangle(\beta)]^p}
{\beta^{p(2p-3)}}\,d\beta,$$
which implies \eqref{leb4} with $p^*=2p-3$.


Now we are in a position to employ Lemma \ref{lem:boi-1} below
with $\alpha=p$, $\beta=0$ and $\kappa=p(2p-3)$.
Then we see that $\langle u \rangle(y)$
blows up in a finite time $y=T_*(\varepsilon)$, provided $pp^* =p(2p-3) \le 1$.
The last condition is equivalent to $1<p \le p_0(5)$.
Therefore, the solution of \eqref{eq.lb1}
blows up in a finite time $T^*(\varepsilon) \le T_*(\varepsilon)$,
if $1<p \le p_0(5)$ and \eqref{bua} hold.
Moreover, we have the upper bound \eqref{life:upper0}
of the life span $T^*(\varepsilon)$.
Therefore, we can conclude the proof of Theorem \ref{bw2},
provided the following lemma is valid.
Although its proof has been given in \cite{KubOht05}, we shall present it in a compact way
in the appendix, for the sake of completeness.

\begin{lm} \label{lem:boi-1}
Let $C_1$, $C_2>0$, $\alpha$, $\beta\ge 0$, $\kappa\le 1$,
$\varepsilon\in (0,1]$, and $p>1$.
Suppose that $f(y)$ satisfies
$$
f(y)\ge C_1\varepsilon^{\alpha},\quad
f(y)\ge C_2\varepsilon^{\beta}\int_{1}^{y}\left(1-\frac{\eta}{y}\right)
\frac{f(\eta)^{p}}{\eta^{\kappa}}\, d\eta,\quad y\ge 1.
$$
Then, $f(y)$ blows up in a finite time $T_*(\varepsilon)$. Moreover,
there exists a constant $C^*=C^*(C_1,C_2,p,\kappa)>0$ such that
$$T_*(\varepsilon)\le \left\{ \begin{array}{ll}
\exp(C^*\varepsilon^{-\{(p-1)\alpha+\beta\}})
& \mbox{if} \hspace{2mm} \kappa=1, \\
C^*\varepsilon^{-\{(p-1)\alpha+\beta\}/(1-\kappa)}
& \mbox{if} \hspace{2mm} \kappa<1.
\end{array} \right.$$
\end{lm}


\section{Small data global existence}

Our first step is to obtain the following estimates for the homogeneous part of the solution
to the problem \eqref{eq.io0}.

\begin{lm} \label{fil1}
Suppose that \eqref{1.3} and \eqref{1.5} hold.
Assume that $ \varphi \in C^1([0,\infty))$, $ \psi \in C^0([0,\infty))$ satisfy
\begin{align} \label{eq.ad2}
|\varphi(r)| \le C_0\, r \, \langle r \rangle^{-(\kappa+2)}, \quad
|\varphi^\prime(r)| + |\psi(r) | \le C_0 \langle r \rangle^{-(\kappa+2)} \ \mbox{for} \ r \ge 0
\end{align}
with some positive constants $C_0$ and $\kappa$.
Then we have
\begin{equation}\label{eq.ue6a}
\left| \int_{ |t-r| }^{t+r} E_-(t,r,y) \left( \psi(y) +\varphi^\prime(y) + w(y)
\varphi(y) \right) dy
\right|
\lesssim \frac{ C_0 r}{\langle r \rangle^2 \langle t-r \rangle^{\kappa}}.
\end{equation}
for  $  t > 0$, $r>0$.
Moreover, for $0<t \le r$ we have
\begin{equation}\label{eq.ue6c}
  \left| E_-(t,r,r-t) \varphi(r-t) \right|
\lesssim
\frac{ C_0 r}{\langle r \rangle^2 \langle t-r \rangle^{\kappa}}.
\end{equation}
\end{lm}

\begin{proof}
We begin with the proof of \eqref{eq.ue6a}.
In the following, let $ t > 0$, $r>0$.
Since $0 \le r-t +y \le 2y$ for $y \ge |t-r|$, from \eqref{eq.le2} we have
\begin{align} \label{eq.zj1h}
|E_-(t,r,y)| \lesssim \langle y \rangle / \langle r \rangle \quad
\mbox{for} \ y \ge |t-r|.
\end{align}
Therefore, by using the assumptions on the data, the left hand side of \eqref{eq.ue6a} is estimated by
\begin{align*}
\langle r \rangle^{-1} \int_{|r-t|}^{t+r} { \langle y \rangle} \left( |\psi(y)| +|\varphi^\prime(y)|
+\langle y \rangle^{-1} |\varphi(y)| \right) dy
 \lesssim
C_0 \langle r \rangle^{-1}
\int_{|r-t|}^{t+r} { \langle y \rangle}^{-\kappa-1} dy,
\end{align*}
which leads to \eqref{eq.ue6a}, because the last integral is estimated as follows:
$$
\int_{|r-t|}^{t+r} { \langle y \rangle}^{-\kappa-1} dy \lesssim
{ \langle t-r \rangle}^{-\kappa},
\quad
\int_{|r-t|}^{t+r} { \langle y \rangle}^{-\kappa-1} dy \lesssim
r{ \langle t-r \rangle}^{-\kappa-1},
$$
and $\langle r \rangle \sim r$ for $r \ge 1$.

Next we prove \eqref{eq.ue6c}, by assuming $0< t \le r$.
From \eqref{eq.zj1h} we have
\begin{align*} 
|E_-(t,r,y) \varphi(y) | \lesssim \frac{C_0 y}{\langle r \rangle \langle y \rangle^{\kappa+1} } \quad
\mbox{for} \ y \ge |t-r|.
\end{align*}
Therefore, we have
$$
\left| E_-(t,r,r-t) \varphi(r-t) \right|
\lesssim
\frac{ C_0 (r - t)}{\langle r \rangle \langle r-t \rangle^{\kappa+1}},
$$
which implies \eqref{eq.ue6c}.
This completes the proof.
\end{proof}

It follows from \eqref{eq.io00} and Lemma \ref{fil1} that
\begin{align} \label{eq.io2}
|u_0(t,r)| \lesssim  \varepsilon r\, \langle r \rangle^{-2} \langle r-t \rangle^{-\kappa} \ \ \mbox{for}\  \ t>0, ~r>0,
\end{align}
provided \eqref{eq.zj1d}  holds, because we have set
$\varphi(r)=r \tilde{f}_0(r)$, $\psi(r)=r \tilde{f}_1(r)$.
This estimate suggests us 
to define the following weighted $L^\infty$-norm:
\begin{equation}
\| u \| =\sup_{ (r,t) \in [0,\infty) \times [0,T] }
\{ r^{-1} \langle r \rangle^2 \langle t-r \rangle^{2p-3}
|u(t,r)| \}.
\end{equation}

Our next step is to consider the integral operator appeared in \eqref{eq.io0}:
\begin{align} \label{eq.io1}
I_-(F)(t,r)=
\frac12 \iint_{\Delta_-(t,r)} E_-(t-\sigma,r,y) F(\sigma, y) dy d\sigma.
\end{align}

\begin{lm}\label{l.ub1}
If $p>p_0(5)=(3+\sqrt{17})/4$, then we have
\begin{equation}\label{eq.ue11}
\| {I_-} (F) \|
\lesssim \| u \|^p
\end{equation}
with $F(t,r)= |u(t,r)|^p/r^{p-1} $, and
\begin{equation}\label{eq.ue12}
\| {I_-} (G) \|
\lesssim \| u -v\| (\|u\|+\|v \|)^{p-1}
\end{equation}
with $G(t,r)= (|u(t,r)|^p-|v(t,r)|^p)/r^{p-1}$.
\end{lm}

\begin{proof}
We begin with the proof of \eqref{eq.ue11}.
For $(y,\sigma) \in \Delta_-(t,r)$ we have $y \ge |t-r-\sigma|$, so that
\eqref{eq.zj1h} yields
$$
E_-(t-\sigma,r,y) \lesssim
\langle r \rangle^{-1} \langle y \rangle \quad \mbox{for} \ (y,\sigma) \in \Delta_-(t,r).
$$
Therefore, using the following estimate
$$
\langle r \rangle^{2p}\langle t-r \rangle^{p(2p-3)} |F(t,r)|
\le r \|u\|^p
$$
in \eqref{eq.io1}, we get
$$
|{I_-} (F)| \lesssim \|u\|^p I(t,r),
$$
where we put
$$
I(t,r) = \iint_{\Delta_-(t,r) }
\frac{ y}{\langle r \rangle \langle y \rangle^{2p-1}
\langle \sigma - y \rangle^{p(2p-3) } } dy d\sigma.
$$

First, suppose $t\ge r$.
To evaluate the integral, we pass to the coordinates
\begin{align} \label{eq.zj0}
 \beta = \sigma-y, \ z=y
\end{align}
and deduce
\begin{align} \label{eq.zj1}
I(t,r)
& \lesssim \int_{r-t}^{t-r} \int_{(t-r-\beta)/2}^{(t+r-\beta)/2}
\frac{1} 
{\langle r \rangle
\langle \beta \rangle^{p(2p-3)} \langle z\rangle^{2p-2}} dz d \beta
\\ \notag
& \hspace{5mm}
+\int_{-t-r}^{r-t} \int_{-\beta}^{(t+r-\beta)/2}
\frac{1} 
{\langle r \rangle \langle \beta \rangle^{p(2p-3)} \langle{z}\rangle^{2p-2}} dz d \beta.
\end{align}

Let $r \ge 1$.
Noting that $t-r-\beta>0$ for $\beta<t-r$, and $-\beta>0$ for $\beta<r-t$, we get
\begin{align*}
\langle r \rangle I(t,r)
& \lesssim \int_{r-t}^{t-r} \frac{1}{\langle \beta \rangle^{p(2p-3)} \langle t-r-\beta\rangle^{2p-3}} d \beta
\\
& \hspace{5mm}
+\int_{-t-r}^{r-t}
\frac{1} 
{\langle \beta \rangle^{p(2p-3)} \langle{\beta}\rangle^{2p-3}} d \beta,
\end{align*}
since $p>p_0(5)>3/2$.
Splitting the integral at $\beta=(t-r)/2$ in the first term of the right hand side, we obtain
\begin{align*}
\langle r \rangle I(t,r)
& \lesssim \frac1{\langle t-r \rangle^{p(2p-3)}} \int_{(t-r)/2}^{t-r} \frac{1}{ \langle t-r-\beta\rangle^{2p-3}} d \beta
\\
& \hspace{5mm} + \frac1 {\langle{t-r}\rangle^{2p-3} }
\int_{-t-r}^{(t-r)/2}
\frac{1} 
{\langle \beta \rangle^{p(2p-3)} } d \beta
\\
& \equiv I_1+I_2.
\end{align*}
Since $p>p_0(5)$ is equivalent to $p(2p-3)>1$, it is clear that
\begin{align} \label{eq.zj2}
\langle t-r \rangle^{(2p-3)}
I_{2} \lesssim 1.
\end{align}
On the other hand, it follows that
\begin{align*}
\langle t-r \rangle^{p(2p-3)}
I_{1} \lesssim  \left\{
\begin{array}{ll}
1 & \mbox{if}\ p>2,
\\
\log (2+t-r) & \mbox{if}\ p=2,
\\
\langle t-r \rangle^{4-2p} & \mbox{if}\ p<2,
\end{array}
\right.
\end{align*}
so that
\begin{align} \label{eq.zj3}
\langle t-r \rangle^{(2p-3)}
I_{1} \lesssim  1,
\end{align}
because $4-2p-(p-1)(2p-3)=1-p(2p-3)<0$ for $p>p_0(5)$.
Combining \eqref{eq.zj2} with \eqref{eq.zj3}, we obtain for $r \ge 1$
\begin{align} \label{eq.zj4}
I(t,r) \lesssim r \langle r \rangle^{-2} \langle{t-r}\rangle^{-(2p-3)}.
\end{align}
When $0<r \le 1$, we see from \eqref{eq.zj1} that
\begin{align*}
\langle r \rangle I(t,r)
& \lesssim r \int_{r-t}^{t-r} \frac{1}{\langle \beta \rangle^{p(2p-3)} \langle t-r-\beta\rangle^{2p-2}} d \beta
+\int_{-t-r}^{r-t}
\frac{1} 
{ \langle{\beta}\rangle^{2p-3}} d \beta
\\
& \lesssim r (\langle t-r \rangle^{-p(2p-3)} + \langle t-r \rangle^{-(2p-2)} )
+ r \langle{t-r}\rangle^{-(2p-3)}
\end{align*}
for $p>p_0(5)>3/2$.
Thus we get \eqref{eq.zj4} for $0<r \le 1$.

Next, suppose $0<t<r$.
Then the change of variables \eqref{eq.zj0} gives
\begin{align*} 
I(t,r)
& \lesssim \int_{-t-r}^{t-r} \int_{-\beta}^{(t+r-\beta)/2}
\frac{1} 
{\langle r \rangle \langle \beta \rangle^{p(2p-3)} \langle{z}\rangle^{2p-2}} dz d \beta
\\
& \lesssim \int_{-t-r}^{t-r}
\frac{1} 
{\langle r \rangle \langle \beta \rangle^{p(2p-3)} \langle{\beta}\rangle^{2p-3}} d \beta,
\end{align*}
since $-\beta>0$ for $\beta<t-r$, and $p>3/2$.
When $r \ge 1$, we get
\begin{align*}
\langle r \rangle I(t,r)
 \lesssim \langle{t-r}\rangle^{-(2p-3)},
\end{align*}
since $p>p_0(5)>3/2$.
On the other hand, when $0<r\le 1$, we use the following estimate:
\begin{align*}
\langle r \rangle I(t,r)
 \lesssim t \langle{t-r}\rangle^{-p(2p-3)-(2p-3)}
 \lesssim r \langle{t-r}\rangle^{-(2p-3)}.
\end{align*}
These estimates leads to \eqref{eq.zj4}, and hence \eqref{eq.ue11} holds.

In order to prove \eqref{eq.ue12}, it suffices to notice the following estimate:
$$
\langle r \rangle^{2p} \langle t-r \rangle^{p(2p-3)} |G(r,t)|
\le p r \| u - v\| (\|u\|+\|v \|)^{p-1},
$$
because the remaining part of the proof is the same as before.
This completes the proof.
\end{proof}

\begin{proof}[Proof of Theorem \ref{ge1}]
If we define a sequence $\{u_n\}_{n=-1}^\infty$ by
\begin{align} \label{eq.io3}
u_{n+1}(t,r) 
=
\varepsilon u_0(t,r)+  I_- ( |u_n|^p/ r^{p-1}) \ \ \mbox{for} \ t>0,~ r>0,
\end{align}
with $u_{-1} \equiv 0$, then \eqref{eq.io2} and Lemma \ref{l.ub1} shows that
it is a Cauchy sequence in
$$
X=\{ u \in C([0,\infty) \times [0, \infty)); ~ \| u \| < \infty \}
$$
for sufficiently small $\varepsilon$.
Thus, we get a solution to the integral equation \eqref{eq.io0}.
This completes the proof.
\end{proof}

\renewcommand{\theequation}{A.\arabic{equation}}
\setcounter{equation}{0}  
\renewcommand{\thelm}{A.\arabic{lm}}
\renewcommand{\thethm}{A.\arabic{theorem}}
\renewcommand{\theprop}{A.\arabic{prop}}
\setcounter{thm}{1}
\section*{Appendix A: Proof of Lemma \ref{lem:boi-1} }
First, we consider the case $\kappa=1$. We put
$$F(z)=(C_1\varepsilon^\alpha)^{-1}
f(\exp(\varepsilon^{-\mu}z)),
\quad \mu=(p-1)\alpha+\beta.$$
Then we have 
\begin{equation*}
F(z)\ge 1,\quad
F(z)\ge C_1^{p-1}C_2\int_{0}^{z}\left(1-e^{-\ve^{-\mu} (z-\zeta)}\right)
F(\zeta)^{p}\, d\zeta, \quad z\ge 0.
\end{equation*}
Since the function $z \mapsto (1-e^{-z})$
is increasing on $[0,\infty)$,  for $0<\varepsilon \le 1$, we obtain
\begin{equation}\label{mo6}
F(z)\ge 1,\quad
F(z)\ge C_1^{p-1}C_2\int_{0}^{z}\left(1-e^{-(z-\zeta)}\right)
F(\zeta)^{p}\, d\zeta, \quad z\ge 0.
\end{equation}
By Lemma \ref{lem:boi-2} below, we can conclude that 
$F(z)$ blows up in a finite time and the desired estimates of the lifespan
hold for the case $\kappa=1$.

Next, we consider the case $\kappa<1$. If we put
$$G(z)=(C_1\varepsilon^\alpha)^{-1}
f(\varepsilon^{-\nu}e^z),
\quad \nu=\frac{(p-1)\alpha+\beta}{1-\kappa},
$$

\newpage

then we have
\begin{equation*}
G(z)\ge 1,\quad
G(z)\ge C_1^{p-1}C_2\ve^{(p-1)\alpha+\beta}
  \int_{0}^{z} \left(1-e^{-(z-\zeta)}\right) \frac{ G(\zeta)^{p}}
{ (\ve^{-\nu} e^\zeta )^{\kappa-1} }
\, d\zeta
\end{equation*}
for $z \ge 0$ and $0<\varepsilon \le 1$, 
which implies \eqref{mo6}, because $e^\zeta \ge 1$ for $\zeta \ge 0$.
Therefore, what we have to do is to show the following lemma.


\begin{lm}\label{lem:boi-2}
Let $C>0$ and $p>1$.
Suppose that $f(t)$ satisfies
\begin{equation}\label{mo8}
f(t)\ge 1,\quad
f(t)\ge C \int_{0}^{t}(1-e^{-(t-\tau)}) f(\tau)^{p}\, d\tau
\end{equation}
for any $t\ge 0$.
Then, $f(t)$ blows up in a finite time.
\end{lm}

\begin{proof}
By \eqref{mo8}, for $t\ge 1$, we have
\begin{eqnarray*}
f(t)
&\ge& 
C \int_{0}^{t}(1-e^{-(t-\tau)}) \, d\tau 
=C(t-1+e^{-t} )
\\
&\ge &C (t-1).
\end{eqnarray*}
We put
\begin{equation}\label{def:A1}
A_1=\exp\left(1+\frac{\log \gamma}{p-1}
+2\sum_{j=1}^{\infty}\frac{\log j}{p^j}\right), \quad
\gamma=\max\{\frac{2}{C(1-e^{-1})},1\}.
\end{equation}
Then, there exists $T_1>1$ such that $f(t)\ge A_1$ for any $t\ge T_1$.

Now, we define sequences $\{A_k\}$ and $\{T_k\}$ by
\begin{equation}\label{def:AT}
A_{k+1}=\frac{A_k^p}{\gamma k^{2}}, \quad
T_{k+1}=T_k+\frac{2}{k^2},
\quad k\in \mathbb{N}.
\end{equation}
Then, for any $k\in \mathbb{N}$, we see that
$f(t)\ge A_k$ for $t\ge T_k$.
Indeed, by \eqref{mo8},
for $t\ge T_k+2k^{-2}$, we have
\begin{eqnarray*}
f(t) &\ge& C \int_{t-k^{-2}}^{t}
(1-	e^{-(t-\tau)}) f(\tau)^{p}\, d\tau \\
&\ge& C A_k^p \int_{0}^{k^{-2}}
(1-e^{-\sigma}) \, d\sigma
\ge \frac{C (1-e^{-1})}{2} A_k^p k^{-2},
\end{eqnarray*}
because $1-e^{-\sigma} \ge (1-e^{-1}) \sigma$ for $0 \le \sigma \le 1$.
Moreover, by \eqref{def:A1} and \eqref{def:AT}, we have
\begin{eqnarray*}
\log A_{k+1} &=&
p^k
\left(\log A_1-\sum_{j=1}^{k}\frac{\log \gamma}{p^j}
-2\sum_{j=1}^{k}\frac{\log j}{p^j}\right) \\
&\ge&
 p^k\left(\log A_1-\frac{\log \gamma}{p-1}
-2\sum_{j=1}^{\infty}\frac{\log j}{p^j}\right)= p^k, \\
T_{k+1} &=& T_1+\sum_{j=1}^{k}\frac{2}{j^2}
\end{eqnarray*}
for any $k\in \mathbb{N}$.
Therefore, $f(t)$ blows up in a finite time.
\end{proof}



\vspace{3mm}

\noindent
{\bf Acknowledgement}

We are grateful to Professor M. Ikeda and Professor M. Sobajima for valuable discussion during the preparation of this work.

\bibliographystyle{amsplain}
\bibliography{DWP30(Arxiv)}

\end{document}